\documentclass[12pt]{amsart}
\usepackage{euscript,amsmath,amsfonts,amssymb,comment}
\usepackage[colorlinks=true, linkcolor=blue, citecolor=blue]{hyperref}
\usepackage{graphicx}

\newcommand{\mH}{\mathcal{H}}

\numberwithin{equation}{section}

\newtheorem{theorem}{Theorem}[section]

\newtheorem{scholium}[theorem]{Scholium}
\newtheorem{lemma}[theorem]{Lemma}

\newtheorem{prop}[theorem]{Proposition}

\usepackage{amsthm}

\makeatletter
\newtheorem*{rep@theorem}{\rep@title}
\newcommand{\newreptheorem}[2]{%
\newenvironment{rep#1}[1]{%
 \def\rep@title{#2 \ref{##1}}%
 \begin{rep@theorem}}%
 {\end{rep@theorem}}}
\makeatother

\newreptheorem{theorem}{Theorem}
\newreptheorem{lemma}{Lemma}

\begin{document}

\title[]{Simplifying $3$-manifolds in $\mathbb{R}^4$}

\author{Ian Agol}
\thanks{Ian Agol is supported by NSF DMS-1105738.}

\author{Michael Freedman}
\thanks{Michael Freedman is supported by Microsoft Research.}

\begin{abstract}
We show that a smooth embedding of a closed $3$-manifold in $S^3\times\mathbb{R}$ can be isotoped so that every generic level divides $S^3\times t$ into two handlebodies (i.e., is Heegaard) provided the original embedding has a unique local maximum with respect to the $\mathbb{R}$ coordinate.  This allows uniqueness of embeddings to be studied via the mapping class group of surfaces and the Schoenflies conjecture is considered in this light.  We also give a necessary and sufficient condition that a $3$-manifold connected summed with arbitrarily many copies of $S^1\times S^2$ embeds in $\mathbb{R}^4$.
\end{abstract}

\maketitle

\section{Introduction}
We work in the smooth category.  Some fundamental questions in geometric topology concern embeddings of $3$-manifolds in $\mathbb{R}^4$.  Several closed $3$-manifolds are known to embed in a homotopy sphere \cite{BudneyBurton} but are not known to embed in $S^4$, so \emph{existence} is related to the smooth 4D Poincar\'{e} conjecture.  The Schoenflies conjecture that every embedded $3$-sphere in $\mathbb{R}^4$ bounds a (smooth) ball is the most famous \emph{uniqueness} question.  This paper sets up some machinery which may be useful for uniqueness questions.  Our main result is a kind of normal form we call a ``Heegaard embedding.''

\begin{reptheorem}{Morse position}
Let $e:M^3\hookrightarrow S^3\times\mathbb{R}$ be a (smooth) embedding of a closed $3$-manifold which is generic in the sense that the composition $\pi\circ e: M^3\overset{e}\hookrightarrow S^3\times\mathbb{R} \xrightarrow{\pi} \mathbb{R}$ is a Morse function.  If $M$ has a unique local maximum, then $e$ is isotopic to an embedding $f: M^3\hookrightarrow S^3\times\mathbb{R}$ so that for all generic levels $t$, $f(M)\cap S^3_t$ is a Heegaard surface for $S^3_t$---that is, $f(M)\cap S^3_t$ cuts $S^3_t$ into two handlebodies.  We call an embedding with this property \emph{Heegaard}.
\end{reptheorem}

\noindent\emph{Remark.} Notice the asymmetry of the hypothesis: $e$ is permitted to have any number of local minima.  Of course, $S^3\times\mathbb{R}$ can be inverted by $t\rightarrow -t$, so having a single local minimum also implies an isotopy to Heegaard position. We also note that if $M^3\cong S^3$, and the Morse function induces a Heegaard splitting of $S^3$ of genus $\leq 3$, then Scharlemann has shown that the embedding is isotopic to a standard 3-sphere \cite{Scharlemann08}.

Section 2 discusses examples of manifolds which cannot have a unique local maximum in Morse position. Section 3 explores the uniqueness of stabilizations of manifolds. Section 4 gives the proof of Theorem \ref{Morse position} and Section 5 explores the uniqueness of the embeddings via the Goeritz group \cite{Goeritz:1933} of the ``middle level'' Heegaard surface.

{\bf Acknowledgement:} We thank Marty Scharlemann for helpful correspondence.

\section{No embedding with unique local maxima}
The main result of the paper discusses embeddings with a unique local maximum. This section is a counterpoint, to demonstrate that
there are embedded codimension-one manifolds for which any Morse embedding must have multiple local maxima.

First, we consider the 3-dimensional case.
\begin{prop}
Any surface in  Morse position in $\mathbb{R}^3$ with a unique local maximum must be a Heegaard surface of $S^3\supset \mathbb{R}^3$.
\end{prop}
\begin{proof}
Consider $\Sigma \subset S^3$ with a unique local maximum with respect to a coordinate function $e:S^3\to [-1,1]$ (and assume that $\Sigma$ does
not meet the north or south poles of $S^3$ with respect to this coordinate). We may build up the complementary regions
of $\Sigma$ by increasing the coordinate function. For small $\epsilon$ so that $e^{-1}([-1,-1+\epsilon])\cap \Sigma =\emptyset$, we see that $e^{-1}([-1,-1+\epsilon])$ is a 3-ball. As we go through a critical point of index $i$ of $\Sigma$, $0\leq i\leq 2$, one of the complementary regions of $\Sigma$ gets an $i$-handle attached, and the other remains unchanged. So one of the complementary regions of $\Sigma$ has a handle decomposition with no $2$-handles, which implies that it must be a handlebody. Sliding the maximum of $\Sigma$ over the north pole of $S^3$ exchanges the roles of the two complementary regions, so we see that both regions must be handlebodies, and thus $\Sigma$ is a Heegaard splitting of $S^3$.
\end{proof}
Thus, any knotted surface in $S^3$ must have multiple local maxima in any Morse embedding.

In four dimensions, a bit less is known.
A result of Scharlemann \cite{Scharlemann85} says that a $2$-sphere with one local minimum and two local maxima is unknotted.
But it is not known if a knotted $2$-sphere $S^2_k$ ($2$-knot) can have a unique local minimum (which immediately implies $\pi_1(\mathbb{R}^4\setminus S^2_k)\cong \mathbb{Z}$).
Even if one restricts to the general $2$-knot with $\pi_1(\mathbb{R}^4\setminus S^2_k) \not\cong \mathbb{Z}$,
we do not know an argument which shows that the boundary of the tubular neighborhood $(S^1\times S^2)_k\hookrightarrow\mathbb{R}^4$ must have multiple local minima with respect to, say, the fourth coordinate $\pi_4$ on $\mathbb{R}^4$.  However,

\begin{theorem}\label{thm:1.2}
If $(S^1\times S^2)_k\hookrightarrow\mathbb{R}^4$ is the boundary of any tubular neighborhood of a $2$-knot $k$ with $\operatorname{deficiency}(k)\leq 0$, then $(S^1\times S^2)_k$ must have more than one local maximum with respect to any coordinate $\pi_4$ on $\mathbb{R}^4$ (which is generic in the sense that $\pi_4|(S^1\times S^2)_k$ is a Morse function).
\end{theorem}

\begin{proof}
One point compactify $\mathbb{R}^4$ to $S^4$; sometimes $S^4$ will be more convenient to work with than $\mathbb{R}^4$.  By definition, $\operatorname{deficiency}(k) = \operatorname{deficiency}(\pi_1(S^4\setminus S^2_k)) = \operatorname{maximum}(g-r)$, where $g$ is the number of generators and $r$ the number of relations in a given presentation of $\pi_1(S^4\setminus S^2_k)$; the maximum is taken over finite presentations. We denote deficiency by $d$.

The critical points of $\pi_4|(S^2\times S^1)_k$ are of two types ``inner'' and ``outer'' according to which side of the embedding of $(S^2\times S^1)_k$ gains a handle.  We call the $2$-knot complement, the \emph{outside}.  Then the outside critical points of index $=k$, $k=0,1,2,3$, determine a 4D handle structure $\mH$ for the outside (together with a single $0$- and $4$-handle coming from $S^4$).  By a rotation of $S^4$ we may assume that the absolute maximum is an inside handle and so does not contribute to $\mH$.  Associated to $\mH$ is a cell complex and inside the $0$- and $1$-cells we identify and then collapse a maximal tree (as is conventional).  This results in a 2-dimensional cell complex with a single vertex, $e_1$ 1-cells, and $e_2$ 2-cells. But $\chi(S^4-\mathcal{N}(S^2_k))=\chi(S^4)-\chi(S^2)+\chi(S^1\times S^2)=0$, so this cell complex must have euler characteristic $0=1-e_1+e_2$. Thus, $1=e_1-e_2 \leq d\leq 0$, a contradiction.

Actually we have shown that there must be at least $1-d$ outer local maxima, corresponding to at least $1-d$ 3-cells in the resulting cell complex.
\end{proof}

Theorem \ref{thm:1.2} is useful as J.\ Levine \cite{Levine78} has constructed $2$-knots of arbitrarily large negative degeneracy.  His most basic example is a $2$-twist spun trefoil whose group is $\langle t, x\mid x^3=1, txt^{-1} = x^{-1}\rangle$.  This group has $d=0$ but proving this requires an insight: The group's Alexander module  $\Lambda / \langle 3, t+1\rangle$, $\Lambda = Z[t,t^{-1}]$ also admits a notion of deficiency: $d_A:= \operatorname{max}(\sharp\text{ gen} - \sharp\text{ relations})$ which upper bounds group deficiency $d_A + 1 \geq d$.  The advantage of passing to $d_A$ is that routine homological algebra can be used to compute $d_A = -1$, whereas group deficiency is generally more opaque.

C.\ Livingston \cite{Livingston05} extended Levine's construction to construct embeddings of contractible manifolds $W$ in $\mathbb{R}^4$ with closed complement $X$ having $d(\pi_1(X))$ arbitrarily small, and such that $\partial W$ is a homology 3-sphere.  By an argument similar to the proof of Theorem \ref{thm:1.2} we obtain:

\begin{theorem}\label{thm:1.3}
For every $N\geq 0$ there is an embedding $e$ of a homology $3$-sphere $\Sigma$ in $\mathbb{R}^4$ so that any embedding $f$ isotopic to $e$ must have at least $N$ local maxima.\qed
\end{theorem}

We have seen that fundamental group deficiency can force local maxima.  What happens if there is no fundamental group at all, as in the Schoenflies problem which addresses (smooth) embeddings $e:S^3\hookrightarrow\mathbb{R}^4$?  Isotoping $e$ to remove all but one local maximum would imply that both closed sides of $e$ have $2$-complex spines---a sort of ``$\frac{1}{4}$--Schoenflies theorem'' as the goal is to show the two sides have $0$-dimensional spines.  The existence of an isotopy to an embedding with a unique local maximum is open.

In Section 6 of \cite{Gompf91}, Gompf produces a ``genus $=4$'' embedding $e:S^3\hookrightarrow\mathbb{R}^4$ with a unique local maximum and minimum and two inside (and two outside) handles each of indices $1$ and $2$.  The \emph{genus} refers to the genus of the surface obtained after the $0$- and $1$-handles and before the $2$- and $3$-handles.  Gompf proves that $e(S^3)$ bounds a $4$-ball (i.e., is isotopically standard); however, the $2$-spine given by the cores of the inside handles reads out the well-known Akbulut-Kirby \cite{AkbulutKirby85} presentation $p$ of the trivial group: $\{x, y \mid xyx = yxy, x^4 = y^5\}$.  Unless $p$ is Andrews-Curtis (AC) trivial---which most experts doubt---this $2$-spine cannot be deformed through $2$-complexes to a point.  The implication for the embedding $e$ is that any isotopy to the round sphere would necessarily pass through embeddings with multiple local maxima (assuming $p\not\equiv\emptyset$ (AC)).

Going back at least to Zeeman's conjecture \cite{Zeeman64} (still open) that a contractible $2$-complex cross interval collapses to a point without expansions, deformations of complexes have been recognized to be most subtle in dimension $2$.  The necessity of local maxima appearing during unknotting isotopies of $S^3$ in $\mathbb{R}^4$ is a manifold analogue of $2$-complex questions such as the AC and Zeeman's conjectures.

\section{Stable equivalence of embeddings}

However, the Schoenflies Conjecture (SC) itself has a fortunate stability which seems to have gone unnoticed.  The question of whether an embedding $e: S^3\hookrightarrow S^4$ is standard is unaffected by the stabilization procedure: ``trade 4D $i$-handles between two sides, for $i = 0$, $1$, $2$, $3$, or $4$.''  Specifically, the \emph{Schoenflies conjecture} states that any embedding $e: S^3\hookrightarrow S^4$ is standard in the equivalent senses:

\begin{enumerate}
    \item $e(S^3)$ bounds a ball to one side
    \item $e(S^3)$ bounds balls on both sides
    \item $e$ is isotopic to a standard position, which can be taken to be the identity map from the equator of $S^4$ to itself.
\end{enumerate}

The phrase ``trading a 4D $i$-handle'' will mean a change in the set-theoretic decomposition of $S^4$ into two pieces which we think of as the closed \emph{inside} and a closed \emph{outside}, respectively.  So if a closed $3$-manifold $M\subset S^4$ divides $S^4$ into $A\cup_M B$ and $h$ is an $i$-handle of $(A, M)$ (or $(B,M)$) it may be reassigned or \emph{traded} to $(B, M)$ (or $(A,M)$), leading to a new decomposition $S^4 = A^\prime\cup_{M^\prime} B^\prime$, where $B^\prime = B\cup h$ (or $A^\prime = A\cup h$).

Let us formalize the induced equivalence relation.  It is an exercise in transversality that any two codimension $1$ submanifolds of $S^4$ cobound a codimension $1$ submanifold $W^4$ of $S^4\times [0,1]$.  A (generic) projection $W \rightarrow [0,1]$ is Morse and induces a handlebody structure on $W$.

\noindent{\bf{Definition.}} Codimension $1$ spheres $S_i^3 \subset S^4$, $i = 0$, $1$, are \emph{stably equivalent} if there is a proper embedding $W\subset S^4\times [0,1]$, so that $\partial_0 W = S_0^3$, $\partial_1 W = -d(S^3_1)$ and so that the projection $W\rightarrow [0,1]$ is Morse but \emph{without} critical points of index $2$; $d$ is an arbitrary diffeomorph $d:S^4\rightarrow S^4$ and the sign denotes reversal of orientation.

\noindent\emph{Note.} We have thrown the diffeomorphism $d$ into the definition because Theorem \ref{thm:1.4} below addresses a diffeomorphism-invariant property of embedded $3$-spheres.  If it were known that $\pi_0(\operatorname{Diff}^+ S^4) = \{e\}$, $d$ would be unnecessary.  But as it is, without the $d$ we would not know that diffeomorphic embeddings were stably equivalent.

\begin{theorem}\label{thm:1.4}
The diffeomorphism types of the closed complementary regions $A$ and $B$ for an embedding $S^3\hookrightarrow S^4$, $S^4 = A\cup_{S^3} B$, depends \emph{only} on the stable equivalence class of the embedding.  In particular, stably equivalent embeddings are diffeomorphic.
\end{theorem}

\begin{proof}
By an isotopy of $W$, arrange that its handles be attached in index order.  Let us focus on a fixed side, $A$.  Depending on whether $0$- and $4$-handles of $W$ are ``inside'' or ``outside,'' their passage corresponds to punctures of $A$ formed (inside, $0$) /removed (inside, $4$) and disjoint $4$-balls formed (outside, $0$) /removed (outside, $4$).  Similarly an outside $1$-handle adds a $1$-handle to $A$ and an inside $3$-handle removes a $1$-handle from $A$ (by deleting it co-core).  Similarly an inside $1$-handle effectively \emph{adds} a trivial $2$-handle (This uses $\pi_1(A) \cong \{e\}$ and ``homotopy implies isotopy.'') and an outside $3$-handle adds a $3$-handle.

Consider the ``halfway'' codimensional submanifold $A_{1/2}\subset S^4\times \frac{1}{2}$, where $\frac{1}{2}$ denotes a generic level, after the $0$- and $1$-handles of $W$ have been attached but before the $3$- and $4$-handles.  Keeping track of maximal trees, it is easy to see the diffeomorphism type of $A_{1/2}$:

\begin{equation}\label{eqn:1}
A_{1/2} \cong A\natural(\natural_s S^2\times B^2)\natural(\natural_r S^1\times B^3),
\end{equation}
where
$$\begin{array}{ccc}
s & = & \sharp(\text{inside } 1\text{-handles}) - \sharp(\text{inside } 0\text{-handles}) \text{, and} \\
r & = & \sharp(\text{outside } 1\text{-handles}) - \sharp(\text{outside } 0\text{-handles}).
\end{array}$$
By symmetry we also have a description of $A_{1/2}$ starting from the closed interior $A^\prime$ of the \emph{other} embedding.

\begin{equation}\label{eqn:2}
A_{1/2}\cong A^\prime\natural(\natural_s S^2\times B^2)\natural(\natural_r S^1\times B^3)
\end{equation}

Four manifolds do not generally obey unique factorization but in this case we will show how to cancel factors and conclude $A = A^\prime$.

Begin with the composed diffeomorphism
\begin{equation}
g: A\natural(\natural_s S^2\times B^2) \natural(\natural_r S^1\times B^3) \rightarrow A^\prime\natural(\natural_s S^2\times B^2)\natural(\natural_r S^1\times B^3).
\end{equation}
$g$ determines an automorphism $\phi$ of the free group $F(y_1,\ldots,y_r)$.  Nielsen moves, which amount to relabelings $y_i\rightarrow y_i^{-1}$ and $1$-handle slides permit $g$ to be replaced by a similar diffeomorphism $g^\prime$ but now inducing the identity on $\pi_1$.  Let $\{\gamma_1,\ldots,\gamma_s\}$ be disjoint sccs representing standard $\pi_1$-generators in the source and $\{g^\prime \gamma_1,\ldots,g^\prime\gamma_s\}$ be their images under $g^\prime$.  ``Homotopy implies isotopy'' implies $\{g^\prime\gamma_1,\ldots,g^\prime\gamma_s\}$ is isotopic to $\{\gamma_1,\ldots,\gamma_s\}$.  Thus the result of framed surgery in \emph{both} domain and range is a diffeomorphism:
\begin{equation}\label{eqn:4}
h: A\natural(\natural_{r+s} S^2\times B^2) \rightarrow A^\prime \natural(\natural_{r+s} S^2\times B^2).
\end{equation}

The domain may be converted back to (a manifold diffeomorphic to) $A$ by attaching $r+s$ $3$-handles to the set of $2$-spheres $\{S_i^2\times \text{pt}, i = 1,\ldots,r+s\}$.  Attaching a corresponding collection of $3$-handles in the image, we obtain a diffeomorphism
\begin{equation}\label{eqn:5}
h^\prime: A\rightarrow A^\prime\natural(\natural_{r+s} S^2\times B^2)\cup (3\text{-handles attached to } h(S^2_i\times \text{pt})\text{, $i = 1,\ldots,r+s$}).
\end{equation}

However, Theorem 1 of \cite{Trace82} states that if a fixed number of $3$-handles are attached to a $1$-connected $4$-manifold and if the boundary after attachment is connected, then the diffeomorphism type of the result does \emph{not} depend on the details of where the $3$-handles were attached.  The hypothesis about connected boundary amounts to homological independence in our case and is easily verified.  Thus the right-hand side of (\ref{eqn:5}) must be diffeomorphic to $A^\prime$, implying $A$ diffeomorphic to $A^\prime$.  Similarly $B$ is diffeomorphic to $B^\prime$.
\end{proof}

\begin{scholium}\label{schm:1.5}
If $A_t$ is a side of $S^4\times t \setminus W$ for a generic level $t$, then $A\cong A_0\cong \text{Ball}^4$, i.e. is standard, iff $A_t$ is a ``standard'' manifold having the form $A_t = \amalg P_i$, each component $P_i \cong (4\text{-ball}) \natural(\natural_{q_i} S^3\times B^1)\natural(\natural_{r_i} S^2\times B^2)\natural_{s_i}(S^1\times B^3)$ for some $q_i$, $r_i$, $s_i\geq 0$.\qed
\end{scholium}

\section{Proof of Theorem \ref{Morse position}}
\noindent{\bf{Definition.}}
An embedding of a connected compact $3$-manifold without boundary $M$ into $\mathbb{R}^4$ is called \emph{Heegaard} iff:
\begin{enumerate}
    \item the fourth coordinate on the embedded $M$ is a generic ordered Morse function (critical points of higher index taking larger values), and
    \item every  generic level set is a Heegaard surface of its level $\mathbb{R}^3\times t$ (after $\mathbb{R}^3\times t$ is compactified to a $3$-sphere $3^3_t$).
\end{enumerate}
We also refer to an embedding $M\hookrightarrow N\times\mathbb{R}$, $N$ a compact $3$-manifold, as \emph{Heegaard} if it obeys the same comditions with $S^3$ replaced by $N$.

\noindent\emph{Note.}
Heegaard embeddings have unique local maxima and local minima.

\begin{theorem}\label{Morse position}
Let $M$ be a connected compact $3$-manifold without boundary admitting an embedding $e:M\hookrightarrow\mathbb{R}^4$ whose fourth coordinate is a Morse function with one local maximum.  Then $e$ is (smoothly) isotopic to a Heegaard embedding $f:M\hookrightarrow\mathbb{R}^4$.
\end{theorem}

\begin{proof}
As explained earlier, we may first isotope $e$ to obtain the ordering condition.  Next we use some ``obvious-sounding'' but subtle $3$-manifold topology to stabilize the Heegaard decomposition of $M$ induced by the fourth coordinate height function.  We will argue that the stabilizations we select come in Morse-canceling pairs so the manifold $M$ is unchanged.  Furthermore, each $1$-handle to be canceled can be delayed in its appearance all the way up to the level of its canceling $2$-handle, implying that the isotopy class of $M\subset\mathbb{R}^4$ is also unchanged.  Here is the non-trivial fact from $3$-manifold topology:

\begin{lemma}\label{lem:2.2}
(Cf. C.\ Frohman \cite{Frohman89}.) Let $H$ be a handlebody (HB) and $\alpha\subset H$ a proper arc.  Suppose $H\setminus\alpha$ is also a HB.  Then $\alpha$ is boundary parallel (i.e., cobounds a bigon $B$ with complementary arc $\beta$ in $\partial H$).
\end{lemma}

\begin{proof}
By induction on genus.  The case $\text{genus}(H)=0$ follows easily.  Suppose Lemma \ref{lem:2.2} holds for $\text{genus}(H)\leq g$ and now consider the case $\text{genus}(H)=g+1$.  $H'=H\setminus\alpha$ is a HB of genus $g+2$, and conversely $H=H'\cup_\gamma 2$-handle, where $\gamma$ is a small linking circle to $\alpha$.  By Jaco's handle addition lemma \cite{Jaco84}, since $\partial H$ is compressible in $H$, $\gamma$ must fail to be disk busting, i.e., $\exists$ an essential disk $\Delta\subset H'$ with $\partial\Delta\cap\gamma=\emptyset$.  Thus, $\Delta\subset H'=H\setminus\alpha\subset H$.  Cutting $H$ along $\Delta$ yields a HB $J$ with $\alpha\subset J$ and $\text{genus}(J)<\text{genus}(H)$.  $J$ is either $H\setminus\Delta$ or the component of $H\setminus\Delta$ containing $\alpha$ if $\Delta$ separates $H$.  By induction $\alpha$ is boundary parallel in $J$, as witnessed by some bigon $B'$.  However, $B'$ is easily deformed (off two or one copy of $\Delta$) to a bigon $B\subset H$, completing the proof.
\end{proof}

We now introduce a lemma that interpolates between systems $S$ and $T$ of proper arcs with HB complements.

\begin{lemma}\label{lem:2.3}
Let $X$ be a compact connected $3$-manifold with boundary $\partial X$ ($\partial X$ is nonempty but not assumed connected).  Let $S$ and $T$ each be families of disjointly and properly embedded arcs in $X$ with HB complements, meaning $X\setminus S$ and $X\setminus T$ are HBs, $H_1$ and $H_2$, respectively, where we abused notation to write $S$ ($T$) for the union of arcs in $S$ ($T$).  Then arcs may be added one at a time to $S$ (i.e., successively deleted from $X$) until at step $k$ a maximum of arcs $\mathbb{U}$ is reached and then arcs are deleted one at a time until $T$ is reached so that at every step $1,\ldots,n$ the complement of the arcs $M\setminus S_i$ is a HB, $S_1=S$, $S_k=\mathbb{U}$, $S_n=T$.  The HBs $H_1$ and $H_2$ are not assumed to have the same genus.
\end{lemma}

\begin{proof}
Corresponding to $S$ and $T$ we produce Morse functions on $X$ $f_1$ and $f_2$, respectively, $f_i: (X,\partial X)\rightarrow ([0,1],0)$, $i = 1,2$.  The two Morse functions $f_i$ have no interior local minima and the descending $1$-manifolds of $f_1$ ($f_2$) are precisely $S$ ($T$).  In handlebody language $f_i$ gives rise to a handle decomposition of $X$ relative to $\partial X$ in which there are no $0$-handles, the $1$-handles have cores $S$ (if $i=1$ and cores $T$ if $i=2$), and the remaining handles of index $2$ and $3$ form the handlebodies $H_1$ or $H_2$, $i = 1$ or $2$.  To compare $f_1$ and $f_2$, we take a generic $1$-parameter family $f_t$, $1\leq t\leq 2$, of functions from $(X, \partial X)$ to $([0,1], 0)$ and go to work simplifying its Cerf diagram (Chapter 1 \cite{HatcherWagoner}).  The Cerf diagram lies in the rectangle $(t,r)\in [1,2]\times[0,1]$ and consists of the critical values of $f_t$ in $[0,1]$ together with an integer $0$, $1$, $2$, or $3$ labelling the index of the critical point.  The diagram contains smooth arcs of critical point transverse to the lines $\{t\times[0,1]\}$ together with a finite number of cusp points where the local expression for $f_t$ requires a cubic term.  Finally, the diagram contains a finite number of vertical arrows which mark handle slides, or in dynamical language ``saddle connections.''  Slides of $1$-handles (in dynamical language descending $1$-manifolds) over each other will be important to us.

While the generic $\{f_t\}$ will include critical points of index $= 0$, a move which amounts to passage through the quartic ``dovetail singularity'' (see p.~25 \cite{HatcherWagoner}) can, in any dimension, be used to remove all index $=0$ critical points from any $1$-parameter family which does not have such critical points at its endpoints.  We do this.  Next we eliminate all the cubic cusps of the form $f_t = -x_1^2 - x_2^3\pm (t - t_0)x_2 + x_3^2$ in local coordinates.  These cusps are the ``birth/death'' points of $(1,2)$-handle pairs.  This is accomplished by sliding all right-pointing ($\pm = -$) cusps past time $t=2$, i.e., postponing the death of all $(1,2)$-handle pairs, and sliding all left-pointing ($\pm = +$) cusps to before time $t=1$, i.e., ``preponing'' all births of $(1,2)$-handle pairs.
\iffalse without local minima and with $S=\{\text{descending }1\text{-manifolds}(f_1)\}$ and $T=\{\text{descending }1\text{-manifolds}(f_2)\}$, and let $f_t$, $0\leq t\leq 1$, be a generic ``Cerf'' family joining $f_0$ to $f_1$.  It is well known that we may assume no $f_t$ has a local minimum (use ``swallowtail'' moves).  Generally, $f_t$ will contain a finite number of cubic cusps (``births/deaths'') and we postpone right-facing (left-facing) cusps past time $=1$ (before time $=0$).\fi
This stabilizes the two Morse functions $f_1$ and $f_2$ with additional canceling $(1,2)$-pairs.  The points of index $=1$ in the Cerf diagram now consists of a family of arcs each proceeding from $t=1$ to $t=2$.  These arcs may cross (in pairs) as the order of the index $=1$ points varies.  At this point, the only events left in the Cerf graphic affecting the $1$-handles (i.e., critical points of index $=1$) are finitely many, $k$, ``saddle  connections,'' or ``handle slides,'' in which a $1$-handle whose critical point is higher passes over a $1$-handle whose critical point is lower.

If there were no $1$-handle slides, the proof would be finished: we would simply see the original descending $1$-manifolds $S$ first stabilized (by adding to $S$ the descending $1$-manifolds of the additional $(1,2)$-pairs at $t=1$) and then destabilized, in a possibly different manner, to arrive at $T$.  At each step in the process the complement of the descending $1$-manifold is a union of $2$- and $3$-handles and therefore a HB.

At first, handle slides appear to be a problem because they do not seem expressible in the language of successively adding and then deleting arcs.  However, there is a convenient translation.  Suppose we follow the Cerf graphic to a time $t_{-}$ when a descending $1$-manifold, the arc $a$, is about to slide over another descending $1$-manifold $b$.  We now cease to follow the graphic but instead substitute a two-step process which emerges on the ``far side'' $t_+$ of the sliding event at time $t$, $t_- = t - \epsilon$ and $t_+ = t + \epsilon$.  Let us denote by $a+b$ the arc after sliding.  (This abbreviated notation does not completely specify the slide, since it does not record the path of the moving end point, but it is adequate for the present explanation.)  The arc $a+b$ is boundary parallel in the HB $X\setminus\{S_{t_-}\}$, and, similarly, the arc $a$ is boundary parallel in the HB $X\setminus\{S_{t_+}\}$.  Both of these boundary parallelisms are witnessed by a single embedded hexagon, $\beta$, that is a $2$-disk whose boundary has been divided into $6$ segments (see Figure \ref{fig:tikz}(a)).  The boundary of $\beta$ consists of three alternating sides running along $\{a,b,a+b\}$ in $\operatorname{interior}(X)$ and the other three alternating sides lying in $\partial X$.  This hexagon may be regarded as a bigon in $6$ ways by considering any $5$ of the $6$ segments as a single boundary arc.  Two of these $6$ ways are important for us.  First, relating $a+b$ to the five complementary sides, and second, relating $a$ to its five complementary sides.
Thus the two-step process relating the descending $1$-manifolds of $f_{t_-}$ and $f_{t_+}$ may be indicated as:
$$\{a,b\}\rightarrow\{a,b,a+b\}\rightarrow\{b,a+b\}.$$

The first bigon shows that $a+b$ is $\partial$-parallel in the HB $X\setminus($descending $1$-manifolds of $f_{t_-})\subset X\setminus (a\cup b)$ from which we conclude $X\setminus($decending $1$-manifolds of $f_{t_-})\cup (a+b)$ is also a HB.  The second bigon (which has the same underlying hexagon as the first) yields a redundant verification of this fact showing that $a$ is $\partial$-parallel in the HB $X\setminus($descending $1$-manifolds of $f_{t_+}) \subset X\setminus (b\cup a+b)$.  In any case, the two arrows above represent ``steps'' of the type claimed in the lemma: first an arc, $a+b$, is added to the set $S$ and then an arc, $a$, is deleted, all the while maintaining the property that the complement of the set of arcs is a HB.  The beginning and ending arc sets are the descending $1$-manifolds of $f_{t_-}$ and $f_{t_+}$, respectively.  The intermediate arc set can be thought of as the union of these; from the perspective of either end, it contains a single additional boundary parallel arc.

This almost proves the lemma, for we have given a procedure to add and then immediately remove one arc from $S_{t_-}$ proximate to each $1$-handle slide of the Cerf diagram.  The ``complement $=$ HB''-condition throughout is preserved.  If this procedure is preceded by the stabilization step and followed by the destabilization step, it comes close to proving the lemma.  The only missing feature is \emph{monotonicity}.  As the lemma is stated, we are to steadily add arcs until a maximum $\mathbb U$ is reached and then steadily delete arcs until $T$ is reached.  To achieve this refinement, it is necessary to clean up the collection $\{\beta_1,\ldots,\beta_k\}$ of $k$-hexagons---one for each $1$-handle slide---so that their interiors are disjoint.

To clarify this point we should review the exact meaning of the standard terminology within $3$-manifold topology of the phrase ``delete an arc.''  What is actually meant is ``delete the interior of a closed regular neighborhood of the arc'' so that compactness is preserved.  With this in mind---that the bigons $\beta$ do not truly have boundary running over these guiding arcs but rather running along a tube surrounding them---we will show that $\{\beta_1,\ldots,\beta_k\}$ can actually be taken to be pairwise disjoint.  This allows the desired reordering of steps; if $\beta_i$ is disjoint from $\beta_j$, $j > i$, the arc-elimination step based on $\beta_i$ can be delayed until after all arcs forming the set $\mathbb U$ have been created.

Disjointness is established as follows: Let $\beta_i$ be the hexagon representing the slide of $a$ over $b$.  Pushing $a$ nearly all the way across $\beta_i$ (see Figure \ref{fig:tikz}(a) $\rightarrow$ Figure \ref{fig:tikz}(b)), $\beta_i$ now lies in a small collar of the arc $\overline{\beta_i}$ consisting of $\partial\beta_i\setminus a$.  Keeping later $\beta_j$, $j > i$, disjoint from $\beta_i$ is fully encoded in the problem of keeping $\beta_j$ disjoint from the arc $\overline{\beta_i}$.  But the later $1$-handle slides may be achieved by ambiently isotoping the attaching region (the ``foot'') of the $1$-handle.  This ambient isotopy carries sheets of interior ($\overline{\beta_i}$) in front of the moving foot, achieving $\beta_j\cap\overline{\beta_i} = \emptyset$ and therefore $\beta_j\cap\beta_i = \emptyset$.  One such sheet is illustrated (Figure \ref{fig:tikz}(c) $\rightarrow$ Figure \ref{fig:tikz}(d)) for a hypothetical slide carrying some (undrawn) $1$-handle over $a+b$.
\iffalse In the lemma statement, deletion is delayed until all arc additions are made.  This final point is achieved by keeping the interiors of later bigons disjoint from earlier ones.

Each handle slide $s_i$ generates a hexagon $h_i$ which we regard as a bigon $\beta_i$ capable of collapsing the initial sliding $a_i$ into a combination of its final position, the arc that was slid over, and the boundary arcs. Since $a_i$ is no longer in the list of $1$-handles after $s_i$, no later slides pass over $a_i$.  Since each $\beta_j$ is constructed in a thin tube around previously generated arcs and $\partial X$, for $j > i$, $\beta_i\cap\beta_j$ is a collection of disjoint ``floating'' arcs, each with one endpoint in $\partial\beta_i$ and the other in the interior.  This means that the two corners of $h_i$ incident on $a_i$ are disjoint from $\beta_j$, $j > i$, and can be used to push the floating arcs of intersection $\{\beta_i\cap\beta_j\}$, $j > i$, off $\beta_i$, where new floating arcs of intersection $\beta_k\cap\beta_j$ for $k < i$ may be generated.  The initial $h_1$ is embedded and by induction we may use such pushes to maintain pairwise interior disjointness of all the hexagons $\{h_i\}$.  Note that the boundaries of the final hexagon may run multiply over a given $a_i$.\fi\end{proof}

\begin{figure}[hbpt]
  \includegraphics{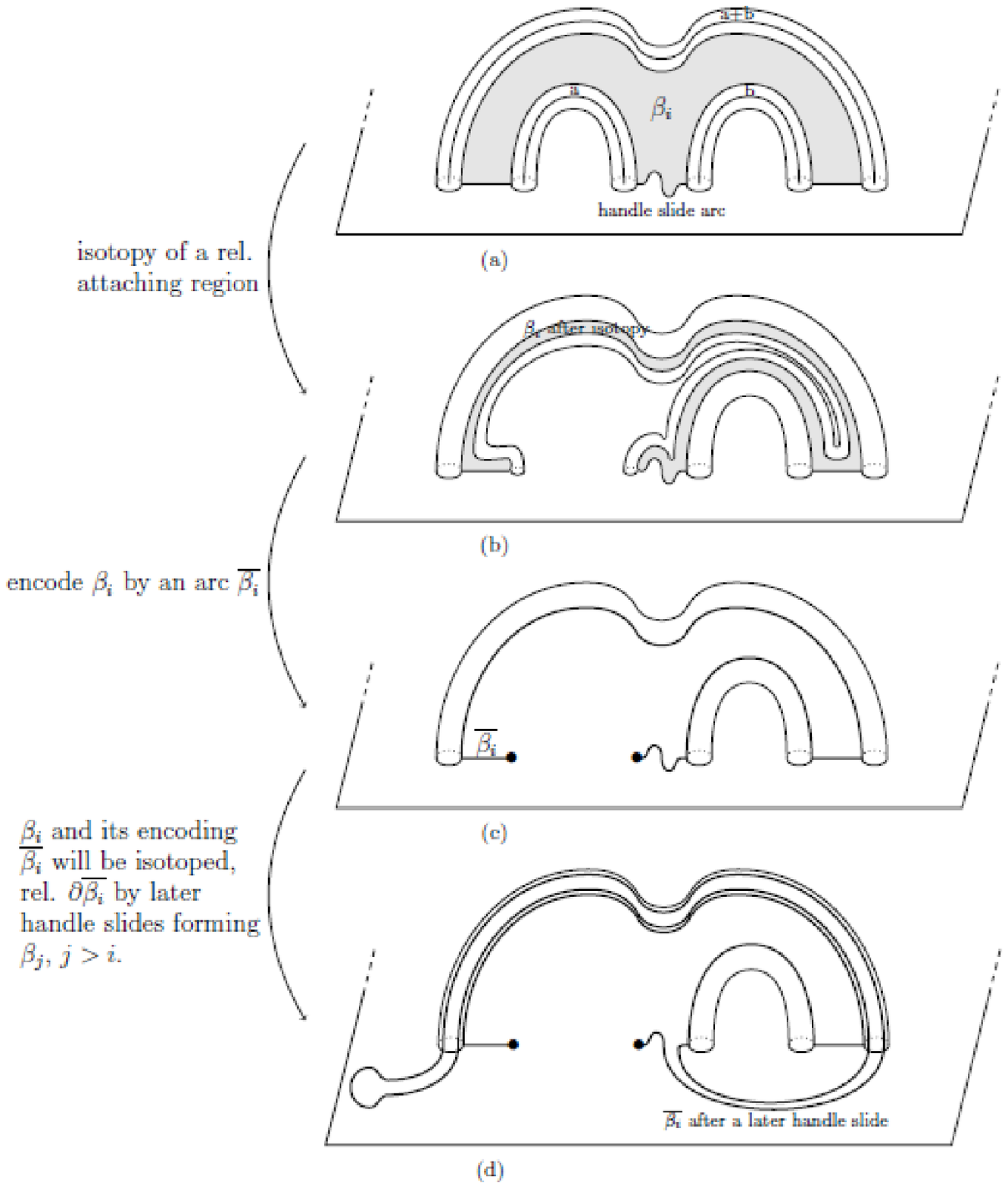}
\caption{}
\label{fig:tikz}
\end{figure}

We return to the proof of Theorem \ref{Morse position}, taking $e:M\hookrightarrow\mathbb{R}^4$ with the fourth coordinate being an ordered Morse function on $M$ with a unique local maximum.  Our plan is to intervene by trading 3D $1$-handles within various levels $\mathbb{R}^3_t\subset\mathbb{R}^4$ to enforce the condition that at each generic level $M\cap\mathbb{R}^3_t\subset\mathbb{R}^3_t\subset S^3_t=R^3_t\cup\infty$ is a Heegaard surface.  It will not be clear until late in the proof that these ``interventions'' can be achieved by an isotopy of $e$; in fact, the authors initially expected that each arc added and later removed would change the topology of $M$ by $\sharp S^1\times S^2$.  However, we discovered that when $e$ has a unique local maximum, Lemma \ref{lem:2.2} permits each arc added to be canceled with a $2$-handle rather than being removed by a ``pinch-off.''  This preserves the topology of $M$ and in fact the isotopy class of $e$.  (In general, when $e$ has multiple local maxima, the $\sharp S^1\times S^2$ factors are inevitable, but Theorem \ref{thm:1.4} shows that they do \emph{no harm} in the case $M\cong S^3$.)

We work from the bottom up to the middle level $S^3_0$, some fixed level between the highest $1$-handle (of $\pi_4\circ e$) and the lowest $2$-handle, and independently from the top down to the middle level.  We encounter a matching problem near $S_0^3$, which is solved by Lemma \ref{lem:2.3}.  In terms of the Heegaard decomposition $M=X\cup Y$ associated to $\pi_4\circ e$, the proof will separately modify and then match up $e|_X$ and $e|_Y$:
\begin{center}
\includegraphics{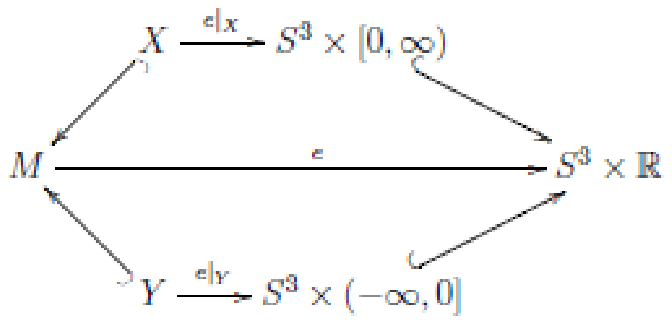}
\end{center}

We start by studying the restriction $e|_Y$ and watching slices of this embedding appear in successive $S^3$-levels as the fourth coordinate $t$ is increased.

In the lowest generic level of $Y$, we see a $2$-sphere.  The next critical point will be index $0$ or $1$, ``in'' or ``out.''  So we see one of four things: a second $2$-sphere ``born'' in or out, or a tube forming again either to the inside or to the outside.  The birth of a $2$-sphere will certainly produce a non-Heegaard level (Heegaard surfaces are connected) and the creation of a tube would also if it were ``knotted.''  Rather than explain narrowly how to maintain the Heegaard property across this second critical point, we may as well treat the general case.  So inductively assume that the generic levels up to $t_{-}$ are Heegaard (we only actually use that the level at $t_{-}$ is Heegaard) and that a critical point at level $t$ is about to destroy the Heegaard property by birthing either a $2$-sphere or a non-boundary parallel, we will call it ``knotted,'' tube.

In both cases we intervene by trading one or more $1$-handles from one side of $M\cap S^3_{t_{-}}$ to the other, i.e. the $1$-handle(s) is (are) deleted from one side and added to the other.  We refer to these intervention $1$-handles as \emph{intervention arcs} (IAs).  The terminology is intended to emphasize the $1$-handle core.  We do so for two reasons.  First, on the deleted side, as in Lemma \ref{lem:2.2}, it is more usual to speak of a boundary parallel arc, not a $1$-handle.  Second, on the side to which the intervention $1$-handle is added, nothing---neither later IAs nor tubes of the evolving $M\cap S_t^3$---need enter the $1$-handle.  Thus these $1$-handles should be pictured as infinitesimally thin and with no internal structure, as clarified below.

Write $M\cap S_t =: A_t \cup_{\Sigma_t} B_t$.  As $t$ increases from $-\infty$ to $0$ (the middle level), we see $2$-spheres and tubes forming on the inside ($A$) and outside ($B$); these are $0$- and $1$-dimensional events, respectively.  In the next paragraphs we explain \emph{how} to find IAs to keep all levels up to $0$ Heegaard; in this paragraph we explain \emph{where} (in what submanifold) the IAs lie.  An inside (outside) IA $\alpha$ introduced at time $t_0 < 0$ (it will always persist to at least $t=0$) will lie in $A_t \setminus \cup_i \alpha_i$, where the union is over all IAs, inside or outside, introduced earlier.  By general position, $\alpha$ may be assumed disjoint from the events ($2$-spheres and tubes) occurring from the time it is introduced up to time $t=0$.  The point to notice is if $A_t$ is enlarged by trading a $1$-handle $\bar{\alpha}$ surrounding, for example, an \emph{outside} IA $\alpha$, we do not consider the ``new material'' in $\operatorname{int}(\bar{\alpha})$ available for future IAs; future IAs are constrained to stay within $A_t$ (not $A_t \cup 1$-handles).  This constraint will not actually make the task (next paragraph) of locating IAs more difficult because the operative assumption is that the manifold in which a tube (or sphere) is forming is a HB (and the purpose of adding IAs is to maintain the HB property.)  The constraint that (say, for $A_t$) we avoid the outside $1$-handles around previous outside IAs merely cuts the HB in which we need to produce IAs into a disjoint union of lower genus HBs.  There is no extra trouble associated with working in these.  Now we turn to the construction of IAs.

In the case of a $2$-sphere birth (i.e., local minimum of $\pi_4\circ e$ at time $t$), a single intervention arc is sent from $S:= M\cap S^3_{t_{-}}$ to meet the $2$-sphere as it appears so that no new local minimum actually occurs.  In the case a tube is forming (due to an index $=1$ critical point of $\pi_4\circ e$ at time $t$), we intervene as follows.  Let $P$ be the side of $S$, minus any previously introduced IAs, in which the tube is forming along an arc $c_0$.  If $P\setminus c_0$ is a HB we consider the tube unknotted and do nothing, i.e. simply let the tube form.  If $P\setminus c_0$ is \emph{not} a HB, let $c_1,\ldots,c_k$ be $1$-handle cores for any handle decomposition (HD) of $P\setminus c_0$ so $P\setminus C=\text{HB}$, $C=\{c_0,\ldots,c_k\}$.  Now apply Lemma \ref{lem:2.3} with $X=P$, $S=\emptyset$, and $T=C$ to find a family of arcs which can be successively added (to the ``deleted set,'' therefore deleted), forming the set $\mathbb{U}$, to arrive at $S_{t_+}$ with not only a tube around $c_0$ formed but $1$-handles formed around all of $\mathbb{U}$.  At each generic level the Heegaard decomposition property has been preserved.  Let $S$ be the totality of arcs added between $-\infty \leq t \leq 0$.

Now turning the fourth coordinate upside down, do the same thing for $X$ that we have just done for $Y$, but in this case the upside-down $X$ has a unique local minimum, so no IAs met $2$-sphere births---a feature we will exploit.  Let $T$ be the totality of deleted IAs from X as we approach the middle level from above, $0_+\leq t\leq\infty$.  Lemma \ref{lem:2.3} addresses the matching problem between $S$ and $T$; it produces a family (whose greatest extent is $\mathbb{U}$ at level $0$) interpolating between $S$ and $T$ with all levels Heegaard.  The required cancelation of IAs is organized by the \emph{Gantt chart} in Figure \ref{fig:1}.

\begin{figure}[hbpt]
    \includegraphics{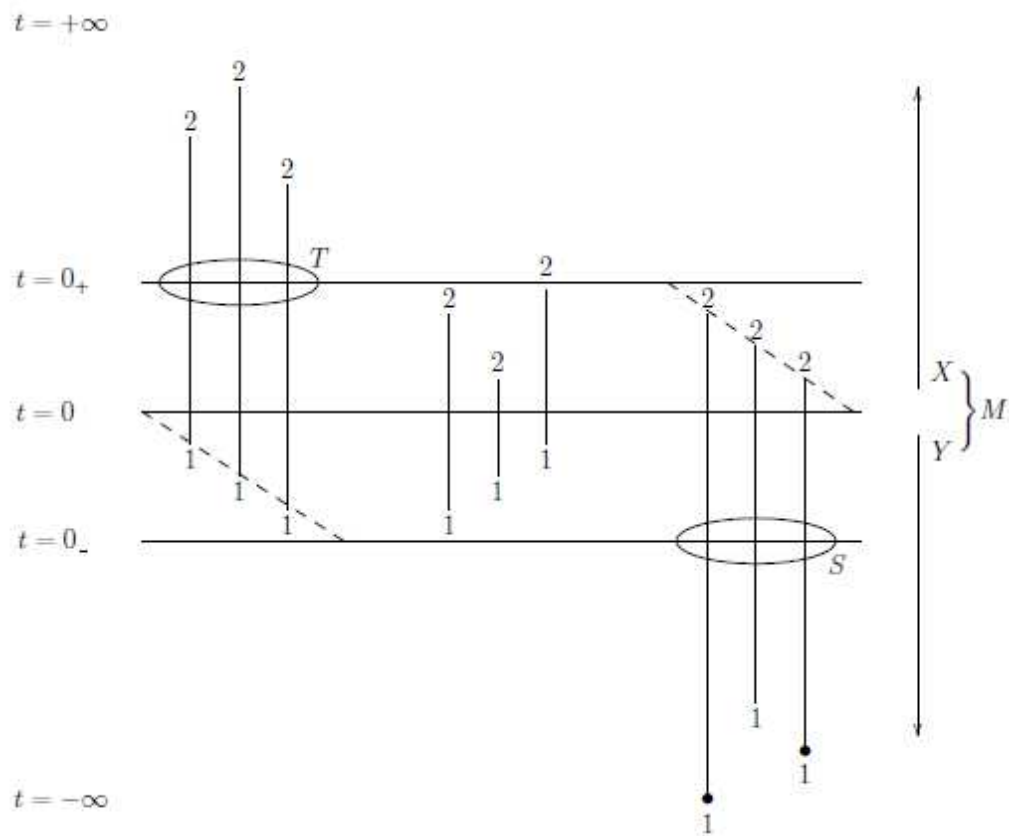}
    \caption{Gantt chart for IAs}
    \label{fig:1}
\end{figure}

Figure \ref{fig:1} shows the time history or \emph{trace} of IAs for an embedded $3$-manifold, realized through our interventions which we will show to be $f: M\hookrightarrow\mathbb{R}^4$, $f$ isotopic to $e$.  The balls at the bottom of some traces indicate the IA prevented a $2$-sphere from forming (running time positively for $Y$ and negatively for $X$).  Note that, per hypothesis, there are no ball markers for $t>0$.  The integers $1$ and $2$ are the index of the \emph{additional} (beyond those of $e$) critical point of the embedding: $1$ for creation of an IA and $2$ for its cancelation by a dual $2$-handle, not yet described in detail.  Logically, the reader may take the view that the $2$ labels tentatively describe pinching off IAs (to the \emph{opposite} side) to produce a ($\sharp S^1\times S^2$)-stabilized embedding $e$.  In the next two paragraphs, after we have explained the alternative cancelation by bigons (on the \emph{same} side), Figure \ref{fig:1} should then be interpreted as describing a new embedding $f: M\hookrightarrow\mathbb{R}^4$ of the original manifold.

Figure \ref{fig:1} shows the arc set $S$ arriving toward $0$ from below and the arc set $T$ arriving toward $0$ from above.  The several arcs in the middle, together with $S\cup T$, constitute the $\mathbb{U}$ in Lemma \ref{lem:2.3}.  Sweeping from $0_{-}$ to $0_+$, we see $S$ expanding to $\mathbb{U}$ at $0$ and then contracting to $T$.  The IA traces in Figure \ref{fig:1} do not indicate which IAs are \emph{inside} and which are \emph{outside}.  The $2$-handles (or bigons) canceling a given IA, which we now describe, are on the same side as the IA.  In fact, the two sides, \emph{in} and \emph{out}, do not interact in any important way; see Scholium \ref{schm:2.4}.

The key observation is that the side containing any given IA $\alpha$ is a HB both immediately before and immediately after the upper endpoint of its trace in Figure \ref{fig:1}.  Thus, Lemma \ref{lem:2.2} tells us that $\alpha$ was boundary parallel in its side.  Let $b$ be a $t$-level bigon in the complement of $M$ and all IAs, $\partial b = \alpha\cup\beta$, where the arc $\beta$ may run over (the $1$-handle sleeves around) IAs as well as $M\cup S_t^3$.  Use $b$ to cancel $\alpha$.  More precisely, use a $2$-handle with core $b$ to cancel the $1$-handle with core $\alpha$.  This manifestly describes a second embedding $f: M\rightarrow\mathbb{R}^4$ of the same $3$-manifold, as the additional $1$- and $2$-handles introduced come in Morse-canceling pairs.

However, a moment's reflection reveals that $f$ is actually isotopic to $e$.  Each intervention arc $\alpha$ (in other language, its surrounding $1$-handle sleeve) is collapsed along a bigon (canceled by a $2$-handle based $1$-surgery).  There is a free parameter: $\alpha$ may be introduced early---as we have done---to preserve the Heegaard property, or later.  Starting with the arc $\alpha$, which is \emph{first canceled}, according to the Gantt chart, postpone its time of introduction until it coincides with the moment it is canceled into its level surface $M\cap S_t^3$ (and perhaps other IA sleeves---these arcs are still present because we have delayed only the first canceled IA $\alpha$).  Now proceed to the second canceled IA $\alpha_2$.  We may similarly delay its introduction to exactly the moment of cancelation.  Because the Gantt chart shows the Heegaard property persisting after $\alpha$ is no longer present, the bigon $b_2$ for $\alpha_2$ will not pass over $\alpha$, so it is harmless that it has been canceled. Proceeding from earliest to latest to be canceled, the introduction times of the IAs ($1$-handles) may be delayed up to the moment they are canceled into the current level surface modified by sleeves around the remaining uncanceled IAs.  This constructs an isotopy from $f$ to $e$.
\end{proof}

\noindent\emph{Remark.}
Symmetry is broken in the proof by canceling handles from bottom to top. In fact, the proof of Theorem \ref{Morse position} breaks down if any of the IAs of $X$ connect disjoint $2$-spheres, corresponding to a non-unique local maximum.  In that case, if one tests the hypothesis of Lemma \ref{lem:2.2} by deleting the IA $\alpha$ (i.e., filling it back into the manifold) just before its upper end point (in the Gantt chart), we see that the resulting complementary region has two boundary components--one a $2$-sphere---and so is not a HB.  Hence, the hypothesis is not satisfied.  Even for homological reasons, it is evident that no bigon $b$ can exist for $\alpha$.

Since the \emph{inside} and \emph{outside} are treated independently in the proof, we have the following scholiums (and extensions):

\begin{scholium}\label{schm:2.4}
Assume $M$ is a closed connected $3$-manifold. Let $e: M\hookrightarrow N^3\times\mathbb{R}$, $N$ any closed $3$-manifold, be an embedding with fourth coordinate an ordered Morse function.
\begin{enumerate}
    \item  If there exists  $t\in\mathbb R$ such that $e(M)\cap (N\times t) \subset N\times t$ is a Heegard surface and if in addition all local maxima of $e$ have height $>t$ and all local minima have height $<t$, then $e$ is isotopic to $f$ in Heegaard position.
    \item If for some generic $t$, $e(M)\cap (N\times t)$ bounds a HB in $N\times t$ to one side, then $e$ is isotopic to $f$ so that, on that side, all generic levels are disjoint unions of HBs.  In particular, $f$ will have no interior local minima or maxima.
    \item If there are two generic levels $t_\text{in}$ and $t_\text{out}$, possibly distinct, each between the highest local minimum and the lowest local maximum, so that $e(M)\cap (N\times t_\text{in}) $ ($e(M)\cap ( N\times t_\text{out})$) bounds a HB in $N\times t$ to the inside (outside), then $e$ is isotopic to $f$ in Heegaard position.
\end{enumerate}
\end{scholium}

\begin{proof}
First, the proof made no use of the levels being $\cong S^3$; they could be a general $N$.
\begin{enumerate}
    \item In the proof of Theorem \ref{Morse position}, a level just below the unique local maximum serves as a (Heegaard genus $=0$)-\emph{initial condition} (ic) below which no local maxima occur and the constructed isotopy between $e$ and $f$ is relative to this ic and supported below it.  The proof is unchanged if at any level $t$, some Heegaard decomposition ic replaces the one of genus $=0$, so long as no local maxima occur lower than $t$.

        In this case, the isotopy produced below $t$ is also relative to the identity at level $t$ (and above) and deforms $e$ to an embedding obeying the Heegaard property \emph{below} level $t$, for $t^\prime \leq t$.  Now turning the fourth coordinate upside down, an isotopy to Heegaard position \emph{above} level $t$ can be found by the same reasoning.  The two isotopies fit together to glue the desired result.
    \item The Gantt chart, Figure \ref{fig:1}, may be thought of as two non-interacting charts, one for the \emph{inside} and one for the \emph{outside}, superimposed on each other.  The creation of ($1$-$2$) pairs, following (1) above, which produces $f$, and the Morse cancellation of such pairs, which produces the isotopy from $e$ to $f$ can both be done on a single side with the claimed result.
    \item Apply (2) first to one side and then to the other to obtain the claimed result.
\end{enumerate}
\end{proof}

The next theorem achieves a standard form stated (for $3$-spheres in $S^3\times\mathbb R$) by Bill Eaton at U.\ C.\ Berkeley in 1980 during a series of lectures whose goal was to prove the Schoenflies conjecture.  We have no written record, but fortunately his statement was recalled to us by Bob Edwards.  With the notation of \ref{schm:2.4} we have:

\begin{theorem}\label{thm:2.5}
(Eaton-Edwards position)  Any smooth embedding $e: M\hookrightarrow N^3\times\mathbb{R}$ whose fourth coordinate is a Morse function with a unique local maximum (note that there is no assumption about local minima) is isotopic to the following folded position.  Specifically, $M$ has some decomposition into a ``chain'' of three submanifolds with boundary with the first and second glued along $\partial_1$ and the second and third glued along $\partial_2$: $M=H_1\cup_{\partial_1} B\cup_{\partial_2} H_2$, where $H_i$ are HBs, $i = 1,2$.  There are embeddings $H_1\overset{i}{\hookrightarrow}N^3\times -1$, $H_2\overset{j}{\hookrightarrow}N^3\times +1$ and $\bar{B}\overset{k}{\hookrightarrow}N^3\times 0$, where the bar indicates reversed orientation, so that if $p \in \partial_1 = H_1\cap B$ and $q\in\partial_2 = B\cap H_2$, $i(p)$ and $k(p)$ have their first three coordinates equal while $j(q)$ and $k(q)$ also have their first three coordinates equal.  Vertically embedded collars $\partial_1\times [-1,0] \subset N\times [-1,0]$ and $\partial_2\times [0,1] \subset N\times [0,1]$ interpolate between the three disjoint images $i(H_1)$, $k(\bar{B})$, and $j(H_2)$ to parameterize an embedding of $M$ in $N\times \mathbb R$.  (Figure \ref{fig:2} displays this ``folded'' embedding.)
\end{theorem}

\begin{figure}
    \includegraphics{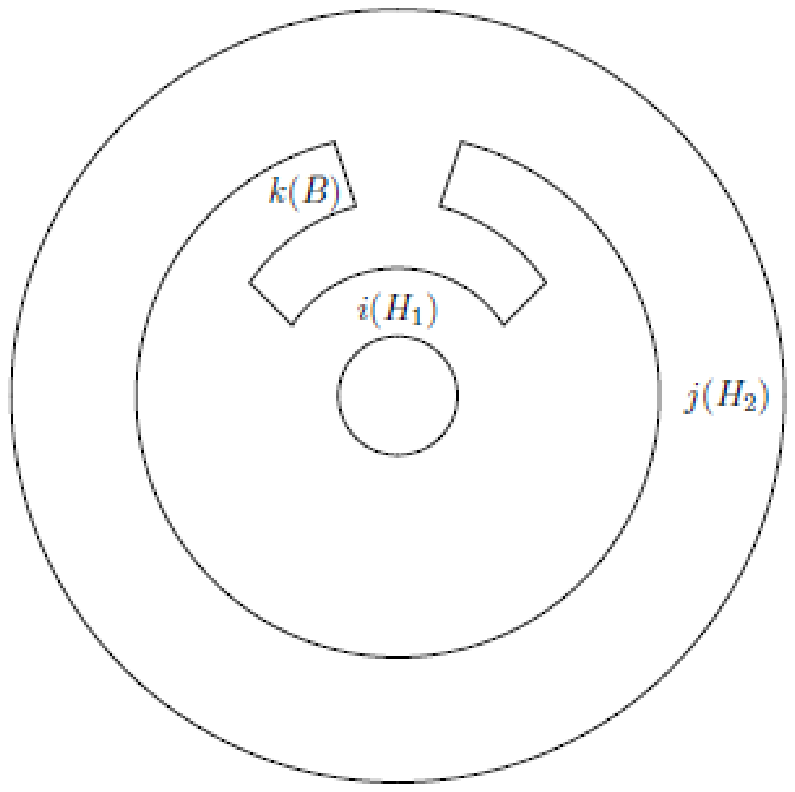}
    \caption{$S^3\times R$}
    \label{fig:2}
\end{figure}

\noindent\emph{Remark.}
Once $e$ is isotoped to Heegaard form, $f: M\rightarrow N\times \mathbb{R}$.  There is \emph{no} intrinsic ordering imposed on the $1$-handles attached to form $Y$ or (reversing the $\mathbb{R}$-coordinate) the $1$-handles attached to form $X$, $M=X\cup Y$.  Each handle is isotopic to a standard handle in a chart---there is no knotting, linking, or nesting to tubes.  All the data for the embedding $f$ is expressible in the gluing data at the middle level, which is the subject of Section \ref{Goeritz} and closely related to the study of the Goeritz group.

\vspace{.1in}
\noindent\emph{Proof of Theorem \ref{thm:2.5}.}
As remarked above, we may reorder the handles of $M$ as induced by the fourth coordinate $\pi_4$ of $f$ to appear in the following order and \emph{batched} as indicated:
\begin{itemize}
    \item $H_1 = 0$-handle and outside $1$-handles of $Y$,
    \item $B =$ (inside $1$-handles of $Y$) $\cup$ (outside $1$-handles of $X$)$_{\pi_4\text{-reversed}}$,
    \item $H_2 =$ ($0$-handle and inside $1$-handles of $X$)$_{\pi_4\text{-reversed}}$.
\end{itemize}
Flattening $f$ so that it is no longer in Morse position but so that the above batches occur simultaneously realizes the desired folded form.\qed

\begin{scholium}\label{schm:2.6}
If in \ref{thm:2.5} $M$ is a rational homology sphere and $N\cong S^3$, we may further arrange that the Morse function $f$ have the same number of handles of each of the four types:
\begin{enumerate}
    \item inner, index $1$;
    \item outer, index $1$;
    \item inner, index $2$;
    \item outer, index $2$.
\end{enumerate}
\end{scholium}

\begin{proof}
An index $(1,2)$ stabilization can be used to increase $\sharp(1)$ and $\sharp(3)$ by one each or $\sharp(2)$ and $\sharp(4)$ by one each.  The homology hypothesis implies $\sharp(1)=\sharp(3)$ and $\sharp(2)=\sharp(4)$, since these handles define presentations for $\pi_1$, or $H_1(\text{ };Q)$ of the two sides, $W$ and $Z$, of $S^4\setminus M$, respectively.  By Alexander duality, $H_1(W;Q)\cong 0\cong H_1(Z;Q)$.
\end{proof}

\section{Relation to the Goeritz group} \label{Goeritz}

%\emph{Definitions}, \emph{history}, \emph{references}, conjectured generators, \emph{abbreviated notation}.
We now give a description of manifolds in Heegaard position in terms of relative mapping class groups.
\noindent If $\Sigma$ bounds a HB $X$, there is a natural embedding $MCG(X)\subset MCG(\Sigma)$, where $MCG(T)$ is $\pi_0(Diff^+(T))$.

Given a genus $g$ Heegaard decomposition of $S^3$, we define the Goeritz group:
$G_g = MCG(X_g)\cap MCG(Y_g)\subset MCG(\Sigma)$, where $S^3 = X\cup_\Sigma Y$ \cite{Goeritz:1933}.
Each element $[\phi] \subset G_g$ induces a diffeomorphism $\Phi: S^3\to S^3$ with $\Phi(X) = X, \Phi(Y)=Y$,
and therefore $\Phi(\Sigma)=\Sigma$ so that $\Phi_{|\Sigma}=\phi \in Diff^+(\Sigma)$. Since $Diff^+(S^3)$ is connected \cite{Cerf}, one has that there is a path of diffeomorphisms $\Phi_t: S^3\to S^3$ such that $\Phi_0=Id_{S^3}$,
and $\Phi_1=\Phi$.
The image $\Phi_t(\Sigma)$ gives an isotopy of the Heegaard surface of $\Sigma$ in $S^3$ which begins and ends in $\Sigma$.
Thus, we may also regard $G_g$ as $\pi_1( Emb(\Sigma,S^3),\Sigma)$, where $Emb(\Sigma, S^3)$ denotes the space of embedded surfaces in  $S^3$ which are homeomorphic to $\Sigma$ (so only depends on the genus $g$).

Suppose $M\subset S^3\times \mathbb{R}\subset S^4$ is in Heegaard form.
We may assume that at level $0\in \mathbb{R}$,
$M\cap (S^3\times \{0\}) =\Sigma$ divides $S^3\cong S^3\times\{0\}$ into handlebodies $S^3= A\cup_\Sigma B$,
and induces a Heegaard splitting $M=X\cup_\Sigma Y$. For each element $[\phi]\in G_g = MCG(A)\cap MCG(B)$, we may realize $\phi: \partial X \to \partial Y$,
and create a new manifold $M'\subset S^3$ by regluing $M'=(X \sqcup Y) / \{ x \simeq \phi(x), x\in \partial X\}$.
Then the manifold $M'$ also has a Heegaard form embedding into $S^3\times \mathbb{R}$,
obtained by shifting the embedding $X\subset S^3\times [0,\infty)$ up by $1$ to $X\hookrightarrow S^3\times [1,\infty)$, keeping $Y$ in its initial position $Y \hookrightarrow S^3\times (-\infty,0]$, and then connecting these by an embedding $\Sigma\times [0,1] \hookrightarrow S^3\times [0,1]$ by  for $x\in \Sigma$,
$(x,t) \mapsto (\Phi_t(x), t)\subset S^3\times [0,1]$, where $\Phi_t$ is defined in the previous paragraph.

In fact, under this operation of changing a manifold $M$ in Heegaard position by the Goeritz group, one may obtain all manifolds in Heegaard position
from a finite collection. The point is that there are only finitely many possible handlebodies $X$ or $Y$ up to the action of $G_g$.
Suppose the handlebody $X$ has $a$ inside handles, and $b$ outside handles, where $a+b=g$. We may choose a 2-sphere $\Theta \subset S^3$
so that $\Theta \cap \Sigma =\theta$ is a circle, and $\theta$ separates $\Sigma$ into surfaces of genus $a$ and $b$ respectively. The handlebody
$X$ is then obtained from $\Sigma$ by compressing a handlebody of genus $a$ inside of $A$ cut off by $\Theta\cap A$, and a handlebody
of genus $b$ inside $B$ cut off by $\Theta\cap B$. By Waldhausen's uniqueness of the genus $g$ Heegaard splittings of $S^3$ \cite{Waldhausen68}, the intersection of $\Sigma$ with the two complementary regions of $S^3-\Theta$ are standard relative splittings of the ball. Thus,
for any other such sphere $\Theta'$ such that $\Theta'\cap \Sigma=\theta'$ cuts $\Sigma$ into surfaces of genus $a, b$, there is an
element of the Goeritz group $[\phi]\in G_g$ such that $\phi(\theta)=\theta'$. Similarly, there is a sphere $\Delta \subset S^3$
cutting $\Sigma$ into subsurfaces of genus $c,d$, where $c$ and $d$ are the number of inside and outside handles of
$Y$ respectively. Let's say $a+c\leq g$. We may then choose a Heegaard embedding of  $\#_{g-a-c} S^2\times S^1$, which
has two disjointly embedded spheres $\Theta', \Delta'$ intersecting the Heegaard splitting  $\Sigma'$ in disjoint curves $\theta'=\Theta'\cap \Sigma', \delta'=\Delta'\cap\Sigma'$. We may assume that $\Sigma'=\Sigma$, and $\Theta'=\Theta$. Then there is an element $[\phi]$ of
the Goeritz group $G_g$ modifying this embedding to the one for the manfiold $M$ as described above, by $\phi(\delta')=\delta=\Delta\cap \Sigma$. Thus, all Heegaard embeddings are obtained
by modification of a standard Heegaard embedding of $\#_{*} S^2\times S^1$ by a Goeritz element.

If the element $[\phi]$ happens to lie in $MCG(X)\cdot  MCG(Y)$, then $M'\cong M$, and we get a re-embedding of $M$ into $S^4$ in Heegaard position. If $M\cong S^3$, then we obtain another smooth Heegaard embedding of $S^3\hookrightarrow S^4$ from $MCG(A)\cap MCG(B)\cap ( MCG(X)\cdot MCG(Y))$. In fact, all
such Heegaard embeddings of $S^3$ of genus $g$ are obtained in this way by Waldhausen's uniqueness theorem for genus $g$ Heegaard
splittings of $S^3$ \cite{Waldhausen68}.
Therefore it seems like an interesting problem to understand the double coset $MCG(A)\cap MCG(B)\cap (MCG(X)\cdot  MCG(Y))$ when $S^3$ is embedded in Heegaard form as a stabilization of the standard embedding in order to understand how to construct all genus $g$ Heegaard embeddings
of $S^3$.
It would be sufficient to understand the highly symmetrical case pictured below where the Heegaard embedding has been stabilized so that $\Sigma$ becomes $genus= 2g$ and the $X ( Y)$ handlebodies are given by compression of the $2g$ loops marked by $x$�s and $y$ �s respectively (see Figure \ref{fig:3}).

It is possible to express part of this information without mentioning Heegaard embedding in the statement.

Refer to Figure \ref{fig:3} to find the $2g$ sccs on $\Sigma_{2g}$ labeled by $x$.  Notice that relative to the inside handlebody the $g$ $x$'s on the left are standard longitudes and the $g$ $x$'s on the left are standard meridians.  Let $\alpha\in MCG(\Sigma)$ take the $2g$ standard meridians to the $x$'s.  We may express the condition \emph{stably embeddable} in terms of the conjugate $G_{2g}^\alpha = \alpha^{-1} G_{2g}\alpha$ of the even genus Goeritz groups.

Some closed $3$-manifolds $M$ which do not embed in $\mathbb{R}^4$ do admit an embedding of $M\setminus\text{pt.}$ into $\mathbb{R}^4$.  An example is the Poincar\'{e} homology $3$-sphere $P= SU(2)/BI$, $BI$ the binary icosahedral group.  $P\setminus\text{pt.}$ arises as a 3D Seifert surface for the $5$-fold-twist-spun trefoil knot \cite{Zeeman65}.  Some closed $3$-manifolds, such as $RP^3$, do not admit even punctured embeddings in $\mathbb{R}^4$.  Such an embedding would yield, by restriction, an embedding of $\mathbb{RP}^2\subset\mathbb{R}^4$ with a nonzero section of its normal bundle; however, it is known that the twisted Euler class of an embedding of $\mathbb{RP}^2$ in $\mathbb{R}^4$ can only assume the values $\pm 2$ \cite{Massey69}.  (See \cite{BudneyBurton} for additional examples.) A notion intermediate between an embedding and a punctured embedding is an embedding of $M\sharp(\sharp_k S^1\times S^2)$ into $\mathbb{R}^4$, for $k\geq 0$, which we call a \emph{stable embedding} of $M$.  There do not seem to be known examples of $3$-manifolds which stably embed in $\mathbb{R}^4$ but do not embed.

\begin{theorem}\label{thm:5.1}
A closed $3$-manifold $M$ stably embeds in $\mathbb{R}^4$ if and only if it has a stabilization $M\sharp(\sharp_k S^1\times S^2)$ which admits a Heegaard decomposition of even genus $2g$ with clutching map $\sigma: \Sigma_{2g} \rightarrow \Sigma_{2g}$ lying in $G^\alpha_{2g}$.
\end{theorem}

\begin{proof}
Given an embedding of $M\sharp(\sharp_k S^1\times S^2)\subset \mathbb{R}^4$, it is possible to perform ambient $0$-surgeries to obtain a new embedding of $M\sharp(\sharp_{k^\prime} S^1\times S^2)\subset \mathbb{R}^4$, $k^\prime\geq k$, with a unique local maximum.  The idea is that from any local maximum of height smaller than the absolute maximum, issue a monotonely rising arc connecting that local maximum to another point on the embedding.  Then an ambient $0$-surgery guided along this arc reduces the number of local maxima (of height less than the absolute maximum) by one.  These $0$-surgeries cause additional stabilization of the embedding but eventually the hypothesis of Theorem \ref{Morse position}, a unique local maximum, will be achieved, and the embedding will be isotopic to a Heegaard embedding.  Once in Heegaard form, the preceding paragraph describes precisely how the induced Heegaard decomposition (HD) is related to Goeritz groups.  If the HD is further stabilized (using both types of stabilization: those that add an $S^1\times S^2$ factor and those that do not), we can reduce to the symmetrical situation where $\Sigma$ arises from $g$ inside and $g$ outside handles from both above ($X$) and below ($Y$).  In this symmetrical case, the clutching map belongs to $G^\alpha_{2g}$.  The converse is immediate.
\end{proof}

\begin{figure}[hbpt]
    \includegraphics{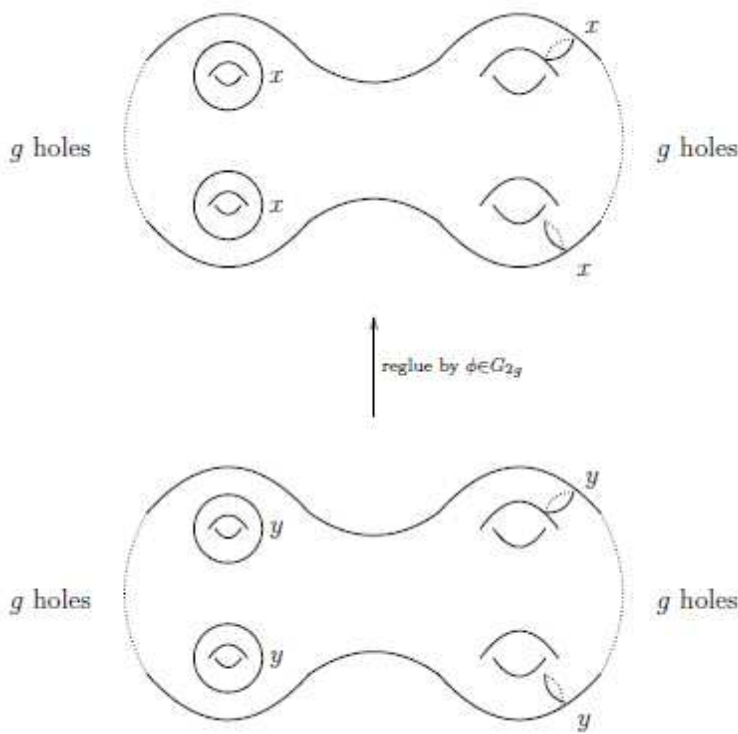}
    \caption{$M = X\cup Y$ has clutching map in a conjugate, $G_{2g}^\alpha$, of the Goeritz group $G_{2g}$.}
    \label{fig:3}
\end{figure}

%\bibliographystyle{../hamsplain}
%\bibliography{../refs}

\bibliography{Simplifying 3-Manifolds in R4 ARXIV_FRIENDLY.bbl}
\end{document}